\documentclass[11pt]{amsart}
\usepackage[numbers,sort&compress]{natbib}
\bibpunct{[}{]}{,}{n}{}{;}
\usepackage{amsmath,amsthm,amsopn,amstext,amscd,amsfonts,amssymb,graphicx,float,caption}
\usepackage{comment}
\usepackage{color, url}
\graphicspath{{images/}}
\usepackage{epsfig}
\usepackage{wrapfig}
\DeclareGraphicsExtensions{.eps}
\usepackage[export]{adjustbox}

\usepackage{mathrsfs}
\usepackage{calrsfs}

\theoremstyle{plain}

\newcommand{\CC}{\mathbb{C}}
\newcommand{\NN}{\mathbb{N}}

\newcommand{\RR}{\mathbb{R}}

\newcommand{\EE}{\mathbb{E}}
\newcommand{\E}{\mathrsfs{E}}
\newcommand{\B}{\mathrsfs{B}}
\newcommand{\F}{\mathrsfs{F}}
\newcommand{\D}{\mathrsfs{D}}
\newcommand{\Lu}{\mathrsfs{L}}
\newcommand{\St}{\mathrsfs{S}}
\newcommand{\st}{\mathfrak{S}}
\newcommand{\LL}{\mathfrak{L}}

\newcommand{\G}{\mathrsfs{G}}
\newcommand{\HH}{\mathrsfs{H}}
\newcommand{\M}{\mathrsfs{M}}
\newcommand{\N}{\mathrsfs{N}}
\newcommand{\U}{\mathrsfs{U}}
\newcommand{\V}{\mathrsfs{V}}

\newtheorem{definition}{\sc Definition}[section]
\newtheorem{teo}{\sc Theorem}[section]
\newtheorem{prop}{\sc Proposition}[section]

\newtheorem{eje}{\sc Example }[section]
\newtheorem{coro}{\sc Corollary}[section]

\newcounter{myremark}
\newenvironment{remark}{\bigskip\noindent   \refstepcounter{myremark}
  \textbf{Remark \themyremark}
   }{\par\bigskip}
\numberwithin{myremark}{section}

\begin{document}

\title[Euler matrices and their algebraic properties]{Euler matrices and their algebraic properties revisited}

\author[Y. Quintana]{Yamilet Quintana$^{(1)}$}
\address{Departamento de Matem\'aticas Puras y Aplicadas,
Edificio Matem\'aticas y Sistemas (MYS), Apartado Postal: 89000, Caracas 1080 A, Universidad Sim\'on Bol\'{\i}var,
Venezuela} \email{yquintana@usb.ve}
\thanks{$(1)\,\,\,$Partially supported by the grants Impacto Caribe (IC-002627-2015) from Universidad del Atl\'antico, Colombia, and DID-USB (S1-IC-CB-003-16) from Decanato de Investigaci\'on y Desarrollo. Universidad Sim\'on Bol\'{\i}var, Venezuela.}

\author[W. Ram\'{\i}rez]{William Ram\'{\i}rez$^{(2)}$}
\address{Departamento de Ciencias B\'asicas, Universidad de la Costa - CUC, Barranquilla, Colombia.}
\email{wramirez4@cuc.edu.co}
\thanks{$(2)\,\,\,$Partially supported by the grant Impacto Caribe (IC-002627-2015) from Universidad del Atl\'antico, Colombia.}

\author[A. Urieles]{Alejandro Urieles$^{(2)}$}
\address{Programa de Matem\'aticas, Universidad del Atl\'antico,
Km 7 V\'{i}a Pto. Colombia, Barranquilla, Colombia.} \email{alejandrourieles@mail.uniatlantico.edu.co}



\begin{abstract}
This paper is concerned with the generalized Euler polynomial matrix $\E^{(\alpha)}(x)$ and the Euler matrix $\E$.  Taking into account some properties of Euler polynomials and numbers, we deduce product formulae for  $\E^{(\alpha)}(x)$ and determine the inverse matrix of $\E$. We establish some explicit expressions for the Euler polynomial matrix  $\E(x)$, which involving  the generalized Pascal, Fibonacci and Lucas matrices, respectively. From these formulae we get some new interesting  identities  involving Fibonacci and Lucas numbers. Also, we provide some factorizations of the Euler polynomial matrix in terms of Stirling matrices, as well as, a connection between the shifted Euler matrices and Vandermonde matrices.
\end{abstract}

\maketitle

\section{Introduction}
\label{intro}

The classical Euler polynomials $E_{n}(x)$ and the generalized Euler polynomials $E_{n}^{(\alpha)}(x)$ of (real or complex) order $\alpha$,  are usually defined as follows (see, for details, \cite{AIK2014,N,SCh2012,SM}):

\begin{equation}
\label{euler1}
\displaystyle\left(\frac{2}{e^{z}+1}\right)^{\alpha}e^{xz} =\displaystyle\sum\limits_{n=0}^{\infty}
E_{n}^{(\alpha)}(x)\frac{z^n}{n!}, \quad |z|<\pi, \quad 1^{\alpha}:=1,
\end{equation}
and
\begin{equation}
\label{euler2}
E_{n}(x):= E_{n}^{(1)}(x), \quad n\in \NN_{0},
\end{equation}
where $ \NN_{0}:= \NN\cup\{0\}$ and $\NN=\{1,2,3,\ldots\}$.

The numbers $E_{n}^{(\alpha)}:= E_{n}^{(\alpha)}(0)$ are called generalized Euler numbers of order $\alpha$, $n\in \NN_{0}$. It is well-known that the classical Euler numbers are defined by the generating function
\begin{equation}
\label{euler3}
\frac{2}{e^{z}+e^{-z}}=\sum_{n=0}^{\infty} \varepsilon_{n}\frac{z^{n}}{n!}.
\end{equation}

The sequence $\{\varepsilon_{n}\}_{n\geq 0}$ counts the numbers of alternating $n$-permutations. Let recall us that a permutation  $\sigma$ of  a set of $n$ elements (or $n$-permutation), is said  alternating if and only if  the $n-1$ differences $\sigma(i+1)-\sigma(i)$ for $i=1,2,\ldots, n-1$ have alternating signs (cf. \cite[p. 258]{C1974}). From \eqref{euler2} and \eqref{euler3} it is easy to check that the connection between the classical Euler numbers and the Euler polynomials is given by the formula
\begin{equation}
\label{euler4}
\varepsilon_{n}=2^{n}E_{n}\left(\frac{1}{2}\right), \quad n\in\NN_{0}.
\end{equation}

So, the numbers $E_{n}:=E_{n}(0)$ also are known in the literature as Euler numbers (cf., e.g., \cite{LS2005,SCh2012}).

The first six generalized Euler polynomials are
$$\begin{aligned}
E_{0}^{(\alpha)}(x)=& \,1,\quad E_{1}^{(\alpha)}(x)=x-\frac{\alpha}{2},\quad E_{2}^{(\alpha)}(x)= x^{2}-\alpha x+\frac{\alpha(\alpha-1)}{4},\\
E_{3}^{(\alpha)}(x)=& \, x^{3}-\frac{3\alpha}{2}x^{2}+\frac{3\alpha(\alpha-1)}{4}x-\frac{3\alpha^{2}(\alpha-1)}{8}, \\
E_{4}^{(\alpha)}(x)=&\, x^{4}-2\alpha x^{3}+\frac{3\alpha(\alpha-1)}{2}x^{2}-\frac{\alpha^{2}(\alpha-3)}{2}x
+\frac{\alpha(\alpha^{3}-6\alpha^{2}+3\alpha-26)}{16},\\
E_{5}^{(\alpha)}(x)=& \, x^{5}-\frac{5\alpha}{2}x^{4}+\frac{5\alpha(\alpha-1)}{2}x^{3}-\frac{5\alpha^{2}(\alpha-3)}{4}x^{2}
+\frac{5\alpha(\alpha-1)(\alpha^{2}-5\alpha-2)}{16}x  \\
&-\frac{\alpha^{2}(\alpha^{3}-10\alpha^{2}+15\alpha+10)}{32}.
\end{aligned}$$

Recent  and  interesting works dealing with these polynomials, Appell and Apostol type polynomials, their properties and applications in several areas as such as combinatorics, number theory,  numerical analysis and  partial differential equations,
can be found by reviewing  the  current literature on this subject. For a broad information on old  literature and
new research trends about these classes of polynomials we strongly recommend to the interested reader see \cite{C1974,HAS2016,HASA2015,HQU2015,LS2005,N,PS2013,QRU,Rio68,SBR2018,SCh2012,SKS2017,SMR2018,SOK2013,SOY2014,SP2003}.

From the generating relation \eqref{euler1}, it is fairly straightforward to deduce the addition formula:
\begin{equation}
\label{euler6}
E_{n}^{(\alpha+\beta)}(x+y)= \sum_{k=0}^{n}\binom{n}{k}E_{k}^{(\alpha)}(x)E_{n-k}^{(\beta)}(y).
\end{equation}

And, it follows also that
\begin{equation}
\label{euler11}
E_{n}^{(\alpha)}(x+1)+ E_{n}^{(\alpha)}(x) =2 E_{n}^{(\alpha-1)}(x).
\end{equation}
Since $E_{n}^{(0)}(x)=x^{n}$, making the substitution $\beta=0$ into  \eqref{euler6} and interchanging $x$ and $y$, we get
\begin{equation}
\label{euler7}
E_{n}^{(\alpha)}(x+y)= \sum_{k=0}^{n}\binom{n}{k}E_{k}^{(\alpha)}(y)x^{n-k}.
\end{equation}
And, as an immediate consequence, we have
\begin{eqnarray}
\label{euler8}
E_{n}(x+y)&=& \sum_{k=0}^{n}\binom{n}{k}E_{k}(y)x^{n-k},\\
\label{euler9}
E_{n}(x)&=&\sum_{k=0}^{n}\binom{n}{k}E_{k}\,x^{n-k}.
\end{eqnarray}
Using \eqref{euler4}, \eqref{euler8} and the well-known relation $E_{n}(1-x)=(-1)^{n}E_{n}(x)$, it is possible to deduce the following  connection formula  between $E_{n}$ and the classical Euler numbers $\varepsilon_{n}$:
\begin{equation}
\label{euler5}
E_{n}=\left\{ \begin{array}{l}-\frac{1}{2^{n}}\sum_{k=0}^{n}\binom{n}{k}\varepsilon_{n-k},\quad \mbox{ if } n \mbox{ is odd},\\
\\
0, \quad \mbox{ if } n \mbox{ is even}.
\end{array}\right.
\end{equation}

Inspired by the article \cite{ZW2006} in which the authors introduce the generalized Bernoulli matrix and  establish some algebraic properties of the Bernoulli polynomial and Bernoulli matrices, in the present article  we focus our attention on the  algebraic and differential properties of the  generalized Euler matrix. It is worthwhile to mention that the authors of \cite{ZW2006} to point out that their proposed methodology can be used for obtaining similar properties in the setting of the generalized Euler matrix. However, the authors of \cite{ZW2006} do not actually write out any proof about  this statement.

The outline of the paper is as follows. Section \ref{sec:1}  has an auxiliary character and provides some background as well as some  results which will be used throughout the paper. Making use of the some identities above, we introduce the generalized Euler matrix  in  Section \ref{sec:2}. Then, we study some interesting particular cases of this matrix, namely, the Euler polynomial matrix, the Euler matrix and the specialized Euler matrix. The main results of this section are Theorems \ref{teogeneuler1}, \ref{teogeneuler2}, \ref{teogeneuler3} and  \ref{teogeneuler4}, because these theorems contain the information concerning the product formula for the Euler matrix, an explicit expression for the inverse matrix of  the specialized Euler  matrix, the factorization of the Euler matrix via the generalized Pascal matrix of first kind, and a useful factorization for the inverse matrix of a particular ``horizontal sliding''  of the Euler polynomial matrix, respectively. Also, some consequences of these results are showed (see for instance, Corollaries \ref{corgeneuler1}, \ref{corgeneuler2}, \ref{corgeneuler3} and \ref{corgeneuler4}).  Section \ref{sec:3} shows several factorizations of the generalized Euler matrix in terms the Fibonacci and Lucas matrices, respectively  (cf. Theorems \ref{teogeneuler5} and \ref{teogeneuler6}). Also, some new identities  involving Fibonacci and Lucas numbers are given in this section. Finally, in Section \ref{sec:5} we provide some factorizations of the Euler polynomial matrix in terms of Stirling matrices, and  the shifted Euler matrices and their connection with Vandermonde matrices are given.

\section{Background and previous results}
\label{sec:1}

Throughout this paper, all matrices are in $M_{n+1}(\RR)$, the set of all $(n+1)$-square matrices over the real field. Also, for $i,j$ any nonnegative integers we adopt the following convention
$$\binom{i}{j}=0, \mbox{ whenever }  j>i.$$

In this section we recall the definitions of the generalized Pascal matrix, the Fibonacci matrix and the Lucas matrix, as well as, some properties of these matrices.

\begin{definition}
\label{defi1}
Let $x$ be any nonzero real number. The generalized Pascal matrix of first kind $P[x]$ is an $(n+1)\times(n+1)$ matrix whose entries are given by (see \cite{CV1993,Z1997}):
\begin{equation}
\label{pascal1}
p_{i,j}(x)=\left\{\begin{array}{l}
\binom{i}{j}x^{i-j}, \quad i\geq j,\\
\\
0, \quad \mbox{otherwise}.
\end{array}\right.
\end{equation}
\end{definition}

In \cite{CV1993,Z1997,Z1998} some properties of the  generalized Pascal matrix of first kind are showed, for example,  its matrix factorization by special summation matrices, its associated differential equation and  its bivariate extensions. The following proposition summarizes some algebraic and differential properties of $P[x]$.

\begin{prop}
Let $P[x]$ be the generalized Pascal matrix of first kind and order $n+1$. Then the following
statements hold.
\begin{enumerate}
\item[(a)] Special value. If the convention $0^{0}=1$ is adopted, then it is possible to define
\begin{equation}
\label{pascal2}
P[0]:= I_{n+1}={\rm diag}(1,1,\ldots,1),
\end{equation}
where $I_{n+1}$ denotes the identity matrix of order $n+1$.
\item[(b)] $P[x]$ is an invertible matrix and its inverse is given by
\begin{equation}
\label{pascal3}
P^{-1}[x]:=\left(P[x]\right)^{-1}= P[-x].
\end{equation}
\item[(c)] \cite[Theorem 2]{CV1993} Addition theorem of the argument. For $x,y\in \RR$ we have
\begin{equation}
\label{pascal4}
P[x+y]= P[x]P[y].
\end{equation}
\item[(d)] \cite[Theorem 5]{CV1993} Differential relation (Appell type polynomial entries). $P[x]$ satisfies the following differential equation
\begin{equation}
\label{pascal5}
 D_{x}P[x]= \LL P[x]= P[x]
 \LL,
\end{equation}
where $D_{x}P[x]$ is the matrix resulting from taking the derivative with respect to $x$ of each entry of $P[x]$ and the entries of the $(n+1)\times(n+1)$ matrix $\LL$ are given by
$$\begin{aligned}
{\rm l}_{i,j}=&\left\{\begin{array}{l}
p_{i,j}'(0), \quad i\geq j,\\
\\
0, \quad \mbox{otherwise},
\end{array}\right.\\
=&\left\{\begin{array}{l}
j+1, \quad i=j+1,\\
\\
0, \quad \mbox{otherwise}.
\end{array}\right.
\end{aligned}$$
\item[(e)] (\cite[Theorem 1]{Z1997}) The matrix $P[x]$ can be factorized as follows.
\begin{equation}
\label{pascal6}
P[x]=G_{n}[x]G_{n-1}[x]\cdots G_{1}[x],
\end{equation}
where $G_{k}[x]$ is the $(n+1)\times(n+1)$ summation matrix given by
$$\begin{aligned}
G_{k}[x]=& \left\{\begin{array}{l}
\begin{bmatrix}
I_{n-k}&0\\
0&S_{k}[x]
\end{bmatrix}, \quad k=1,\ldots, n-1,\\
\\
S_{n}[x],  \quad k=n,
\end{array}\right.
\end{aligned}$$
being $S_{k}[x]$ the $(k+1)\times(k+1)$ matrix whose entries $S_{k}(x;i,j)$ are given by
$$\begin{aligned}
S_{k}(x;i,j)=&\left\{\begin{array}{l}
x^{i-j}, \quad j\leq i,\\
\\
0, \quad  j>i,
\end{array}\right. \quad (0\leq i,j\leq k).
\end{aligned}$$
\end{enumerate}
\end{prop}

Another necessary structured  matrices in what follows, are the Fibonacci and Lucas matrices. Below, we recall the definitions of each one of them.

\begin{definition}
\label{defi2}
Let $\{F_{n}\}_{n\geq 1}$ be the Fibonacci sequence, i.e.,  $F_{n}=F_{n-1}+F_{n-2}$ for $n\geq 2$ with initial conditions $F_{0}=0$ and $F_{1}=1$. The Fibonacci matrix $\F$ is an $(n+1)\times(n+1)$  matrix whose entries are given by \cite{LKL}:
\begin{equation}
\label{fibo2}
f_{i,j}=\left\{\begin{array}{l}
F_{i-j+1}, \quad i-j+1\geq 0,\\
\\
0, \quad i-j+1< 0.
\end{array}\right.
\end{equation}
\end{definition}

Let $\F^{-1}$ be the inverse of $\F$ and denote by $\tilde{f}_{i,j}$ the entries of $\F^{-1}$. In \cite{LKL} the authors obtained the following explicit expression for  $\F^{-1}$.

\begin{equation}
\label{fibo1}
 \begin{aligned}
\tilde{f}_{i,j}=&\left\{\begin{array}{l}
1, \quad i=j,\\
\\
-1, \quad i=j+1, j+2,\\
\\
0, \quad \mbox{otherwise}.
\end{array}\right.
\end{aligned}
\end{equation}

\begin{definition}
\label{defi3}
Let $\{L_{n}\}_{n\geq 1}$ be the Lucas sequence, i.e.,  $L_{n+2}=L_{n+1}+L_{n}$ for $n\geq 1$ with initial conditions $L_{1}=1$ and $L_{2}=3$. The Lucas matrix $\Lu$ is an $(n+1)\times(n+1)$  matrix whose entries are given by \cite{ZZ2007}:

\begin{equation}
\label{lucas}
l_{i,j}=\left\{\begin{array}{l}
L_{i-j+1}, \quad i-j\geq 0,\\
\\
0, \quad \mbox{otherwise}.
\end{array}\right.
\end{equation}
\end{definition}

Let $\Lu^{-1}$ be the inverse of $\Lu$ and denote by $\tilde{l}_{i,j}$ the entries of $\Lu^{-1}$. In \cite[Theorem 2.2]{ZZ2007} the authors obtained the following explicit expression for  $\Lu^{-1}$.

\begin{equation}
\label{lucas1}
 \begin{aligned}
\tilde{l}_{i,j}=&\left\{\begin{array}{l}
1, \quad i=j,\\
\\
-3, \quad i=j+1, \\
\\
5(-1)^{i-j}2^{i-j-2}, \quad i\geq j+2,\\
\\
0, \quad \mbox{otherwise}.
\end{array}\right.
\end{aligned}
\end{equation}

For $x$  any nonzero real number, the following relation between the matrices $P[x]$ and $\Lu$ was stated and proved in \cite[Theorem 3.1]{ZZ2007}.
\begin{equation}
\label{lucas2}
P[x]=\Lu \G[x] =\HH[x]\Lu,
\end{equation}
where  the entries of the $(n+1)\times(n+1)$  matrices $\G[x]$ and $\HH[x]$ are given by
$$\begin{aligned}
g_{i,j}(x)=& x^{-j-1}\left[x^{i+1} \binom{i}{j} -3x^{i}\binom{i-1}{j}+ 5(-1)^{i+1} 2^{i-1}m_{i-1, j+1}\left(\frac{x}{2}\right)\right],\\
\\
h_{i,j}(x)=&x^{-j-1}\left[x^{i+1} \binom{i}{j} -3x^{i}\binom{i}{j+1}+ (-1)^{j+1}\frac{5x^{i+j+2}}{2^{j+3}} n_{i+1, j+3}\left(\frac{2}{x}\right)\right],
\end{aligned}$$
respectively, with
$$\begin{aligned}
m_{i,j}(x):=& \left\{\begin{array}{l}
\sum_{k=j}^{i}(-1)^{k}\binom{k}{j}x^{k}, \quad i\geq j,\\
\\
0, \quad i<j,
\end{array}\right.
\end{aligned}  \quad \mbox{ and } \quad
\begin{aligned}
n_{i,j}(x):=& \left\{\begin{array}{l}
\sum_{k=j}^{i}(-1)^{k}\binom{i}{k}x^{k}, \quad i\geq j,\\
\\
0, \quad i<j.
\end{array}\right.
\end{aligned}$$

\section{The generalized Euler matrix}
\label{sec:2}
\begin{definition}
\label{def3}
The generalized $(n+1)\times(n+1)$ Euler matrix $\E^{(\alpha)}(x)$ is defined by
\begin{equation}
\label{euler10}
E^{(\alpha)}_{i,j}(x)=
\left\{\begin{array}{l}
\binom{i}{j} E^{(\alpha)}_{i-j}(x), \quad i\geq j,\\
\\
0, \quad \mbox{otherwise}.
\end{array}\right.
\end{equation}
While, $\E(x):= \E^{(1)}(x)$ and $\E:= \E(0)$ are called the Euler polynomial matrix and the Euler matrix, respectively. In the particular case $x=\frac{1}{2}$, we call the matrix $\EE:=\E\left(\frac{1}{2}\right)$ specialized Euler matrix.
\end{definition}

It is clear that \eqref{euler11} yields the following matrix identity:
\begin{equation}
\label{geneuler5}
\E^{(\alpha)}(x+1)+ \E^{(\alpha)}(x)= 2\E^{(\alpha-1)}(x).
\end{equation}

Since $\E^{(0)}(x)=P[x]$, replacing $\alpha$ by $1$ in \eqref{geneuler5} we have
\begin{equation}
\label{euler20}
\E(x+1)+ \E(x)= 2P[x].
\end{equation}

Then, putting $x=0$ in \eqref{euler20} and taking into account \eqref{pascal2}, we get
$$\E(1)+ \E= 2I_{n+1}.$$
Analogously,
$$\E+ \E(-1)= 2P\left[-\frac{1}{2}\right].$$

From \eqref{euler4} it follows that the entries of the specialized Euler matrix $\EE$ are given by
\begin{equation}
\label{euler12}
e_{i,j}=
\left\{\begin{array}{l}
\binom{i}{j} 2^{j-i}\varepsilon_{i-j}, \quad i\geq j,\\
\\
0, \quad \mbox{otherwise}.
\end{array}\right.
\end{equation}

From \eqref{euler5} it follows that the entries of the Euler matrix $\E$ are given by
\begin{equation}
\label{euler18}
E_{i,j}=
\left\{\begin{array}{l}
\binom{i}{j} E_{i-j}, \quad i>j \mbox{ and } i-j  \mbox{ odd},\\
\\
1, \quad i=j,\\
\\
0, \quad \mbox{otherwise}.
\end{array}\right.
\end{equation}

The next  result is an immediate consequence of Definition \ref{def3} and the addition formula \eqref{euler6}.

\begin{teo}
\label{teogeneuler1}
The generalized Euler matrix $\E^{(\alpha)}(x)$ satisfies the following product formula.
\begin{equation}
\label{geneuler1}
\E^{(\alpha+\beta)}(x+y)= \E^{(\alpha)}(x)\,\E^{(\beta)}(y)= \E^{(\beta)}(x)\,\E^{(\alpha)}(y)= \E^{(\alpha)}(y)\, \E^{(\beta)}(x).
\end{equation}
\end{teo}

\begin{proof}
We proceed as in the proof of  \cite[Theorem 2.1]{ZW2006}, making the corresponding modifications. Let $A_{i,j}^{(\alpha,\beta)}(x,y)$ be the $(i,j)$-th entry of the matrix product $ \E^{(\alpha)}(x)\,\E^{(\beta)}(y)$, then by the addition formula \eqref{euler6} we have
$$\begin{aligned}
A_{i,j}^{(\alpha,\beta)}(x,y)=&\sum_{k=0}^{n}\binom{i}{k}E_{i-k}^{(\alpha)}(x)\binom{k}{j}E_{k-j}^{(\beta)}(y)\\
=&\sum_{k=j}^{i}\binom{i}{k}E_{i-k}^{(\alpha)}(x)\binom{k}{j}E_{k-j}^{(\beta)}(y)\\
=& \sum_{k=j}^{i}\binom{i}{j}\binom{i-j}{i-k}E_{i-k}^{(\alpha)}(x)E_{k-j}^{(\beta)}(y)\\
=&\binom{i}{j}\sum_{k=0}^{i-j}\binom{i-j}{k}E_{i-j-k}^{(\alpha)}(x)E_{k}^{(\beta)}(y)\\
=& \binom{i}{j}E_{i-j}^{(\alpha)}(x+y),
\end{aligned}$$
which implies the first equality of \eqref{geneuler1}. The second and third equalities of \eqref{geneuler1} can be derived in a similar way.
\end{proof}

\begin{coro}
\label{corgeneuler1}
Let $(x_{1},\ldots,x_{k})\in \RR^{k}$. For $\alpha_{j}$ real or complex parameters, the Euler matrices $\E^{(\alpha_{j})}(x)$ satisfies the following product formula, $j=1,\ldots, k$.
\begin{equation}
\label{geneuler2}
\E^{(\alpha_{1}+\alpha_{2}+\cdots+\alpha_{k})}(x_{1}+x_{2}+\cdots+x_{k})= \E^{(\alpha_{1})}(x_{1})\,\E^{(\alpha_{2})}(x_{2})\,\cdots\, \E^{(\alpha_{k})}(x_{k}).
\end{equation}
\end{coro}

\begin{proof}
The application of induction on $k$ gives the desired result.
\end{proof}

If we take $x=x_{1}=x_{2}=\cdots=x_{k}$ and $\alpha=\alpha_{1}=\alpha_{2}=\cdots=\alpha_{k}$, then we obtain  the following simple formula for the powers of the generalized Euler matrix, and consequently, for  the powers of the Euler polynomial and  Euler matrices.

\begin{coro}
\label{corgeneuler2}
The  generalized Euler matrix $\E^{(\alpha)}(x)$ satisfies the following identity.
\begin{equation}
\label{geneuler3}
\left(\E^{(\alpha)}(x)\right)^{k}= \E^{(k\alpha)}(kx).
\end{equation}
In particular,
\begin{equation}
\label{geneuler4}
\begin{aligned}
\left(\E(x)\right)^{k}=& \E^{(k)}(kx),\\
\E^{k}=& \E^{(k)}.
\end{aligned}
\end{equation}
\end{coro}

\begin{remark}
Note that Theorem \ref{teogeneuler1} and  Corollaries \ref{corgeneuler1} and \ref{corgeneuler2} are  respectively, the analogous of Theorem 2.1 and Corollaries 2.2 and 2.3 of \cite{ZW2006} in the setting of Euler matrices.
\end{remark}

Let $\D$ be the $(n+1)\times(n+1)$ matrix whose entries are defined by
\begin{equation}
\label{euler13}
d_{i,j}=\left\{\begin{array}{l}
(1+(-1)^{i-j})\binom{i}{j} 2^{j-i-1}, \quad i\geq j,\\
\\
0, \quad \mbox{otherwise}.
\end{array}\right.
\end{equation}

\begin{teo}
\label{teogeneuler2}
The inverse matrix of the specialized Euler matrix $\EE$ is given by
$$\EE^{-1}=\D.$$
Furthermore,
$$\left[\E^{(k)}\left(\frac{k}{2}\right)\right]^{-1}= \D^{k}.$$
\end{teo}

\begin{proof}
Taking into account \eqref{euler4} and  \eqref{euler12}, it is possible to deduce
$$ \sum_{k=0}^{n}\frac{(1+(-1)^{k})}{2}\binom{n}{k} 2^{n-k}E_{n-k}\left(\frac{1}{2}\right)=\sum_{k=0}^{n}\frac{(1+(-1)^{k})}{2}\binom{n}{k}\varepsilon_{n-k}= \delta_{n,0},$$
where $\delta_{n,0}$ is the Kronecker delta (cf., e.g., \cite[pp. 107-109]{Rio68}). So, we obtain that the $(i,j)$-th entry of the matrix product $\D\EE$ may be written as
$$
\sum_{k=j}^{i} \binom{i}{k}\frac{(1+(-1)^{i-k})}{2}2^{k-i}\binom{k}{j}E_{k-j}\left(\frac{1}{2}\right) $$
$$\begin{aligned}
=&\binom{i}{j}2^{j-i}\sum_{k=j}^{i}\binom{i-j}{k-j}\frac{(1+(-1)^{i-k})}{2}2^{k-j}E_{k-j}\left(\frac{1}{2}\right)\\
=&\binom{i}{j}2^{j-i}\sum_{k=0}^{i-j}\binom{i-j}{k} \frac{(1+(-1)^{i-j-k})}{2}2^{k}E_{k}\left(\frac{1}{2}\right)
=&\binom{i}{j}2^{j-i}\delta_{i-j,0},
\end{aligned}$$
and consequently, $\D\EE=I_{n+1}$. Similar arguments allow to show that  $\EE\D=I_{n+1}$, and hence $\EE^{-1}=\D$.

Finally, from the identity $\EE^{-1}=\D$ and   \eqref{geneuler4} we see that
$$\left[\E^{(k)}\left(\frac{k}{2}\right)\right]^{-1}= \left(\EE^{k}\right)^{-1}=\left(\EE^{-1}\right)^{k}=\D^{k}.$$
This last chain of equalities finishes the proof.
\end{proof}

It is worthwhile to mention that the calculation of $\EE^{-1}$ strongly depends  on  the use of inverse relations derived from exponential generating functions (cf. \cite[Chap. 3, Sec. 3.4]{Rio68}). This tool can be applied in order to determine $\EE^{-1}$, but it not works for determining of $\E^{-1}$. This fact  and \eqref{euler5} suggest that methodology proposed in \cite{ZW2006} does not suffice to finding an explicit formula for $\E^{-1}$.

The next result establishes the relation between the generalized Euler  matrix and the generalized Pascal matrix of first kind.

\begin{teo}
\label{teogeneuler3}
The generalized Euler matrix $\E^{(\alpha)}(x)$ satisfies the following relation.
\begin{equation}
\label{euler141}
\E^{(\alpha)}(x+y)= \E^{(\alpha)}(x)P[y]= P[x]\E^{(\alpha)}(y)= \E^{(\alpha)}(y)P[x].
\end{equation}
In particular,
\begin{equation}
\label{euler14}
\E(x+y)=P[x]\E(y)=P[y]\E(x),
\end{equation}
\begin{equation}
\label{euler17}
\E(x)=P[x]\E,
\end{equation}
\begin{equation}
\label{euler15}
\E\left(x+\frac{1}{2}\right)=P[x]\EE,
\end{equation}
and
\begin{equation}
\label{euler19}
\E=P\left[-\frac{1}{2}\right]\EE.
\end{equation}
\end{teo}

\begin{proof}
The substitution $\beta=0$ into \eqref{geneuler1} yields
$$\E^{(\alpha)}(x+y)= \E^{(\alpha)}(x)\,\E^{(0)}(y)= \E^{(0)}(x)\,\E^{(\alpha)}(y)= \E^{(\alpha)}(y)\, \E^{(0)}(x).$$
Since $\E^{(0)}(x)=P[x]$, we obtain
$$\E^{(\alpha)}(x+y)=P[x]\E^{(\alpha)}(y).$$
A similar argument allows to show that $\E^{(\alpha)}(x+y)=\E^{(\alpha)}(x)P[y]$ and \\$\E^{(\alpha)}(x+y)= \E^{(\alpha)}(y)P[x]$.

Next, the  substitution $\alpha=1$ into \eqref{euler141} yields \eqref{euler14}. From the substitutions $y=0$ and $y=\frac{1}{2}$ into  \eqref{euler14}, we obtain  the relations \eqref{euler17} and \eqref{euler15}, respectively.

Finally, the substitution $x=-\frac{1}{2}$ into \eqref{euler15} completes the proof.
\end{proof}

\begin{remark}
Note that the relation \eqref{euler14} is the analogous of \cite[Eq. (13)]{ZW2006} in the context of Euler polynomial matrices and, the counterpart of  \eqref{euler17}  is   \cite[Eq. (14)]{ZW2006}. However, the relation \eqref{euler15} is slightly different from \cite[Eq. (14)]{ZW2006}, since it involves an Euler polynomial matrix with ``shifted argument'' and the specialized Euler matrix. More precisely, the relation \cite[Eq. (14)]{ZW2006} reads as
$$\B(x)=P[x]\B,$$
consequently, this relation expresses to the Bernoulli polynomial matrix $\B(x)$ in terms of the matrix product between  the generalized Pascal matrix of first kind $P[x]$ and the Bernoulli matrix $\B$. While, on the left hand side of \eqref{euler15} appears an Euler polynomial matrix with ``shifted argument'', and the matrix product on the right hand side  of  \eqref{euler15} contains to the specialized Euler matrix $\EE$.
\end{remark}

The following example shows the validity of Theorem \ref{teogeneuler3}.

\begin{eje}
Let us consider  $n=3$. It follows from the definition \ref{def3} that
$$\EE=\begin{bmatrix}
1 & 0 & 0 & 0\\
0 & 1 & 0 & 0\\
-\frac{1}{4} & 0 & 1 & 0\\
0 & -\frac{3}{4} & 0 & 1 \end{bmatrix} \quad \mbox{ and }\quad \E\left(x+\frac{1}{2}\right)=
\begin{bmatrix}
1 & 0 & 0 & 0\\
x & 1 & 0 & 0\\
x^2-\frac{1}{4} & 2x & 1 & 0\\
x^3-\frac{3}{4}x & 3x^2-\frac{3}{4} & 3x & 1
\end{bmatrix}.$$

On the other hand, from \eqref{euler15} and a simple computation we have
\begin{eqnarray*}
\E\left(x+\frac{1}{2}\right)&=&
\underbrace{\begin{bmatrix}
1 & 0 & 0 & 0\\
x & 1 & 0 & 0\\
x^2 & 2x & 1 & 0\\
x^3 & 3x^2 & 3x & 1 \end{bmatrix}}_{P[x]}\underbrace{\begin{bmatrix}
1 & 0 & 0 & 0\\
0 & 1 & 0 & 0\\
-\frac{1}{4} & 0 & 1 & 0\\
0 & -\frac{3}{4} & 0 & 1 \end{bmatrix}}_{\EE} =\begin{bmatrix}
1 & 0 & 0 & 0\\
x & 1 & 0 & 0\\
x^2-\frac{1}{4} & 2x & 1 & 0\\
x^3-\frac{3}{4}x & 3x^2-\frac{3}{4} & 3x & 1
\end{bmatrix}.
\end{eqnarray*}
\end{eje}

The next theorem follows by a simple computation.

\begin{teo}
\label{teogeneuler4}
The inverse of the Euler polynomial matrix $\E\left(x+\frac{1}{2}\right)$ can be expressed as
\begin{equation}
\label{euler16}
\left[\E\left(x+\frac{1}{2}\right)\right]^{-1}= \EE^{-1}P[-x]=\D P[-x].
\end{equation}
In particular,
\begin{equation}
\label{euler21}
\E^{-1}= \D P\left[\frac{1}{2}\right].
\end{equation}
\end{teo}

\begin{proof}
Using \eqref{pascal3}, \eqref{euler15} and Theorem \ref{teogeneuler2} the relation \eqref{euler16} is deduced. The substitution $x=-\frac{1}{2}$ into \eqref{euler16} yields \eqref{euler21}.
\end{proof}

\begin{eje}
Let us consider  $n=3$. From the definition \ref{def3} and a standard computation we obtain
\begin{eqnarray*}
\left[\E\left(x+\frac{1}{2}\right)\right]^{-1}&=&\begin{bmatrix}
1 & 0 & 0 & 0\\
x & 1 & 0 & 0\\
x^2-\frac{1}{4} & 2x & 1 & 0\\
x^3-\frac{3}{4}x & 3x^2-\frac{3}{4} & 3x & 1
\end{bmatrix}^{-1}=\begin{bmatrix}
1& 0& 0& 0\\
-x& 1& 0& 0\\
x^{2}+\frac{1}{4}& -2x& 1& 0\\
-x^{3}-\frac{3}{4}x& 3x^{2}+\frac{3}{4}& -3x& 1
\end{bmatrix}.
\end{eqnarray*}

On the other hand, from \eqref{euler16}  we have
\begin{eqnarray*}
\left[\E\left(x+\frac{1}{2}\right)\right]^{-1}&=&
\underbrace{\begin{bmatrix}
1 & 0 & 0 & 0\\
0 & 1 & 0 & 0\\
\frac{1}{4} &0 & 1 & 0\\
0 & \frac{3}{4} & 0 & 1 \end{bmatrix}}_{\D}\underbrace{\begin{bmatrix}
1 & 0 & 0 & 0\\
-x & 1 & 0 & 0\\
x^2 & -2x & 1 & 0\\
-x^3 & 3x^2 & -3x & 1 \end{bmatrix}}_{P[-x]}= \begin{bmatrix}
1& 0& 0& 0\\
-x& 1& 0& 0\\
x^{2}+\frac{1}{4}& -2x& 1& 0\\
-x^{3}-\frac{3}{4}x& 3x^{2}+\frac{3}{4}& -3x& 1
\end{bmatrix}.
\end{eqnarray*}

Hence, when  $x=-\frac{1}{2}$, we get
$$\E^{-1}= \underbrace{\begin{bmatrix}
1 & 0 & 0 & 0\\
0 & 1 & 0 & 0\\
\frac{1}{4} &0 & 1 & 0\\
0 & \frac{3}{4} & 0 & 1 \end{bmatrix}}_{\D}\underbrace{\begin{bmatrix}
1 & 0 & 0 & 0\\
\frac{1}{2} & 1 & 0 & 0\\
\frac{1}{4} & 1 & 1 & 0\\
\frac{1}{8} & \frac{3}{4}& \frac{3}{2} & 1 \end{bmatrix}}_{P\left[\frac{1}{2}\right]}= \begin{bmatrix}
1& 0& 0& 0\\
\frac{1}{2}& 1& 0& 0\\
\frac{1}{2}& 1& 1& 0\\
\frac{1}{2}& \frac{3}{2}&\frac{3}{2}& 1
\end{bmatrix}.$$
\end{eje}

At this point, an apart mention deserves the  recent  work \cite{IRS} since it states an explicit formula to the inverse matrix of the $q$-Pascal matrix plus one in terms of the $q$-analogue of the Euler matrix $\E$.

As a consequence of the relations \eqref{pascal6}, \eqref{lucas2}, and Theorems \ref{teogeneuler3} and \ref{teogeneuler4}, we obtain the following corollaries.

\begin{coro}
\label{corgeneuler3}
The Euler polynomial matrix $\E\left(x+\frac{1}{2}\right)$ and its inverse can be factorized by summation matrices as follows.
$$\begin{aligned}
\E\left(x+\frac{1}{2}\right)=& G_{n}[x]G_{n-1}[x]\cdots G_{1}[x] \EE,\\
\\
\left[\E\left(x+\frac{1}{2}\right)\right]^{-1}=& \D G_{n}[-x]G_{n-1}[-x]\cdots G_{1}[-x].
\end{aligned}$$
In particular,
$$\begin{aligned}
\E=& G_{n}\left[-\frac{1}{2}\right]G_{n-1}\left[-\frac{1}{2}\right]\cdots G_{1}\left[-\frac{1}{2}\right] \EE,\\
\\
\E^{-1}=&\D G_{n}\left[\frac{1}{2}\right]G_{n-1}\left[\frac{1}{2}\right]\cdots G_{1}\left[\frac{1}{2}\right].
\end{aligned}$$
\end{coro}

\begin{coro}
\label{corgeneuler4}
For $x$  any nonzero real number, the Euler polynomial matrix $\E\left(x+\frac{1}{2}\right)$ and its inverse can be factorized, respectively, in terms of the Lucas matrix $\Lu$ and its inverse as follows.
$$\begin{aligned}
\E\left(x+\frac{1}{2}\right)=& \Lu \G[x]\EE= \HH[x]\Lu \EE ,\\
\\
\left[\E\left(x+\frac{1}{2}\right)\right]^{-1}=& \D (\G[x])^{-1} \Lu^{-1}= \D\Lu^{-1} (\HH[x])^{-1} .
\end{aligned}$$
In particular,
$$\begin{aligned}
\E=& \Lu \G\left[-\frac{1}{2}\right]\EE=\HH\left[-\frac{1}{2}\right]\Lu \EE,\\
\\
\E^{-1}=&\D \left(\G\left[-\frac{1}{2}\right]\right)^{-1} \Lu^{-1}= \D \Lu^{-1} \left(\HH\left[-\frac{1}{2}\right]\right)^{-1}.
\end{aligned}$$
\end{coro}

We end this section showing others identities, which can be easily deduced from the content of this paper. So, we will omit the details of their proofs.
$$\begin{aligned}
 D_{x}\E(x+y)=&\LL P[x]\E(y),\\
 &\\
D_{x}\E(x)=&\LL P[x]\E,\\
 &\\
 D_{x}\E\left(x+\frac{1}{2}\right)=&\LL P[x]\EE,\\
 &\\
  D_{x}\left[\E\left(x+\frac{1}{2}\right)\right]^{-1}=& \D\LL P[-x].
\end{aligned}$$

\section{Generalized Euler polynomial matrices via Fibonacci and Lucas matrices}
\label{sec:3}

For $0\leq i,j\leq n$ and $\alpha$ a real or complex number, let $\M^{(\alpha)}(x)$ be the $(n+1)\times(n+1)$ matrix whose entries are given by (cf. \cite[Eq. (18)]{ZW2006}):

\begin{equation}
\label{euler28a}
\tilde{m}^{(\alpha)}_{i,j}(x)= \binom{i}{j}E^{(\alpha)}_{i-j}(x)-\binom{i-1}{j}E^{(\alpha)}_{i-j-1}(x)-\binom{i-2}{j}E^{(\alpha)}_{i-j-2}(x).
\end{equation}
We denote $\M(x)=\M^{(1)}(x)$ and $\M=\M(0)$.

Similarly, let $\N^{(\alpha)}(x)$ be the $(n+1)\times(n+1)$ matrix whose entries are given by (cf. \cite[Eq. (32)]{ZW2006}):
\begin{equation}
\label{euler28b}
\tilde{n}^{(\alpha)}_{i,j}(x)= \binom{i}{j}E^{(\alpha)}_{i-j}(x)-\binom{i}{j+1}E^{(\alpha)}_{i-j-1}(x)-\binom{i}{j+2}E^{(\alpha)}_{i-j-2}(x).
\end{equation}
We denote $\N(x)=\N^{(1)}(x)$ and $\N=\N(0)$.

From the definitions of $\M^{(\alpha)}(x)$ and $\N^{(\alpha)}(x)$, we see that
$$\begin{array}{l}
\tilde{m}^{(\alpha)}_{0,0}(x)= \tilde{m}^{(\alpha)}_{1,1}(x)= \tilde{n}^{(\alpha)}_{0,0}(x)= \tilde{n}^{(\alpha)}_{1,1}(x)= E^{(\alpha)}_{0}(x)=1,\\
\\
\tilde{m}^{(\alpha)}_{0,j}(x)=\tilde{n}^{(\alpha)}_{0,j}(x)=0, \quad j\geq 1,\\
\\
\tilde{m}^{(\alpha)}_{1,0}(x)= \tilde{n}^{(\alpha)}_{1,0}(x)= E^{(\alpha)}_{1}(x)- E^{(\alpha)}_{0}(x)=x-\frac{\alpha}{2}-1,\\
\\
\tilde{m}^{(\alpha)}_{1,j}(x)=\tilde{n}^{(\alpha)}_{1,j}(x)=0, \quad j\geq 2,\\
\\
\tilde{m}^{(\alpha)}_{i,0}(x)= \tilde{n}^{(\alpha)}_{i,0}(x) = E^{(\alpha)}_{i}(x)-E^{(\alpha)}_{i-1}(x)-E^{(\alpha)}_{i-2}(x), \quad i\geq 2.
\end{array}$$

For $0\leq i,j\leq n$ and $\alpha$ a real or complex number, let $\Lu_{1}^{(\alpha)}(x)$ be the $(n+1)\times(n+1)$ matrix whose entries are given by
\begin{equation}
\label{euler28c}
\hat{l}^{(\alpha,1)}_{i,j}(x)= \binom{i}{j}E^{(\alpha)}_{i-j}(x)-3\binom{i-j}{j}E^{(\alpha)}_{i-j-1}(x)+
5\sum_{k=j}^{i-2}(-1)^{i-k}2^{i-k-2}\binom{k}{j}E^{(\alpha)}_{k-j}(x).
\end{equation}
We denote $\Lu_{1}(x)=\Lu_{1}^{(1)}(x)$ and $\Lu_{1}=\Lu_{1}(0)$.

Similarly, let $\Lu_{2}^{(\alpha)}(x)$ be the $(n+1)\times(n+1)$ matrix whose entries are given by
\begin{equation}
\label{euler28d}
\hat{l}^{(\alpha,2)}_{i,j}(x)= \binom{i}{j}E^{(\alpha)}_{i-j}(x)-3\binom{i}{j+1}E^{(\alpha)}_{i-j-1}(x)+
5\sum_{k=j+1}^{i}(-1)^{k-j}2^{k-j-2}\binom{i}{k}E^{(\alpha)}_{i-k}(x).
\end{equation}
We denote $\Lu_{2}(x)=\Lu_{2}^{(1)}(x)$ and $\Lu_{2}=\Lu_{2}(0)$.

From the definitions of  $\Lu_{1}^{(\alpha)}(x)$ and  $\Lu_{2}^{(\alpha)}(x)$, we see that
$$\begin{array}{l}
\hat{l}^{(\alpha,1)}_{i,i}(x)=  \hat{l}^{(\alpha,2)}_{i,i}(x)=1, \quad i\geq 0,\\
\\
\hat{l}^{(\alpha,1)}_{0,j}(x)= \hat{l}^{(\alpha,2)}_{0,j}(x)=0, \quad j\geq 1,\\
\\
\hat{l}^{(\alpha,1)}_{1,0}(x)= \hat{l}^{(\alpha,2)}_{1,0}(x)= E^{(\alpha)}_{1}(x)- 3E^{(\alpha)}_{0}(x)=x-\frac{\alpha}{2}-3,\\
\\
\hat{l}^{(\alpha,1)}_{1,j}(x)= \hat{l}^{(\alpha,2)}_{1,j}(x)=0, \quad j\geq 2,\\
\\
\hat{l}^{(\alpha,1)}_{i,0}(x)= E^{(\alpha)}_{i}(x)-3E^{(\alpha)}_{i-1}(x)+
5\sum_{k=0}^{i-2}(-1)^{i-k}2^{i-k-2}E^{(\alpha)}_{k}(x), \quad i\geq 2,\\
\\
\hat{l}^{(\alpha,2)}_{i,0}(x)= E^{(\alpha)}_{i}(x)-3iE^{(\alpha)}_{i-1}(x)+
5\sum_{k=1}^{i}(-1)^{k}2^{k-2}\binom{i}{k} E^{(\alpha)}_{i-k}(x), \quad i\geq 2,\\
\\
\hat{l}^{(\alpha,1)}_{i,1}(x)= iE^{(\alpha)}_{i-1}(x)-\frac{7(i-1)}{2}E^{(\alpha)}_{i-2}(x)+
5\sum_{k=1}^{i-2}(-1)^{i-k}2^{i-k-2}kE^{(\alpha)}_{k-1}(x), \quad i\geq 3.
\end{array}$$

The following results show several factorizations of $\E^{(\alpha)}(x)$ in terms of Fibonacci and Lucas matrices, respectively.

\begin{teo}
\label{teogeneuler5}
The generalized Euler polynomial matrix $\E^{(\alpha)}(x)$  can be factorized in terms of the Fibonacci matrix $\F$ as follows.
\begin{equation}
\label{euler22}
\E^{(\alpha)}(x)= \F \M^{(\alpha)}(x),
\end{equation}
or,
\begin{equation}
\label{euler22a}
\E^{(\alpha)}(x)= \N^{(\alpha)}(x)\F.
\end{equation}
In particular,
\begin{equation}
\label{euler23a}
\F\M(x)=\E(x)=\N(x)\F,
\end{equation}
\begin{equation}
\label{euler23}
\F\M=\E=\N\F,
\end{equation}
and
\begin{equation}
\label{euler24}
\F\M\left(\frac{1}{2}\right)=\EE=\N\left(\frac{1}{2}\right)\F.
\end{equation}
\end{teo}

\begin{proof}
Since the relation \eqref{euler22} is equivalent to $\F^{-1}\E^{(\alpha)}(x)=  \M^{(\alpha)}(x)$, it is possible to follow the proof given in \cite[Theorem 4.1]{ZW2006},  making the corresponding modifications, for obtaining \eqref{euler22}. The relation \eqref{euler22a} can be obtained using a similar procedure. The relations \eqref{euler23a}, \eqref{euler23} and \eqref{euler24} are straightforward consequences of \eqref{euler22} and \eqref{euler22a}.
\end{proof}

Also, the relations \eqref{euler22} and \eqref{euler22a} allow us to deduce the following identity:

$$\M^{(\alpha)}(x)=\F^{-1}\N^{(\alpha)}(x)\,\F.$$
As a consequence of Theorems \ref{teogeneuler4} and \ref{teogeneuler5}, we can derive simple factorizations for the inverses of the polynomial matrices $\M\left(x+\frac{1}{2}\right)$ and $\N\left(x+\frac{1}{2}\right)$:

\begin{coro}
\label{teogeneuler6}
The  inverses of the polynomial matrices $\M\left(x+\frac{1}{2}\right)$ and $\N\left(x+\frac{1}{2}\right)$  can be factorized as follows.
\begin{equation}
\label{euler24a}
\left[\M\left(x+\frac{1}{2}\right)\right]^{-1}=\D P[-x]\F,
\end{equation}
\begin{equation}
\label{euler24b}
\left[\N\left(x+\frac{1}{2}\right)\right]^{-1}=\F \D P[-x].
\end{equation}
In particular,
\begin{equation}
\label{euler25}
 \M^{-1}= \D P\left[\frac{1}{2}\right]\F, \quad \mbox{ and }  \quad  \N^{-1}= \F\D P\left[\frac{1}{2}\right],
\end{equation}
\begin{equation}
\label{euler26}
\left[\M\left(\frac{1}{2}\right)\right]^{-1}= \D\F, \quad \mbox{ and }  \quad \left[ \N\left(\frac{1}{2}\right)\right]^{-1}= \F\D.
\end{equation}
\end{coro}

An analogous reasoning as used in the proof of Theorem \ref{teogeneuler5}   allows us to prove the results below.

\begin{teo}
\label{teogeneuler8}
The generalized Euler polynomial matrix $\E^{(\alpha)}(x)$  can be factorized in terms of the Lucas matrix $\Lu$ as follows.
\begin{equation}
\label{euler33}
\E^{(\alpha)}(x)= \Lu \Lu_{1}^{(\alpha)}(x),
\end{equation}
or,
\begin{equation}
\label{euler34}
\E^{(\alpha)}(x)= \Lu_{2}^{(\alpha)}(x)\Lu.
\end{equation}
In particular,
\begin{equation}
\label{euler34a}
\Lu \Lu_{1}(x)=\E(x)=\Lu_{2}(x)\Lu,
\end{equation}
\begin{equation}
\label{euler34b}
\Lu\Lu_{1}=\E=\Lu_{2}\Lu,
\end{equation}
and
\begin{equation}
\label{euler34e}
\Lu\Lu_{1}^{\left(\frac{1}{2}\right)}(x)=\EE=\Lu_{2}^{\left(\frac{1}{2}\right)}(x)\Lu.
\end{equation}
\end{teo}

Also, the relations \eqref{euler33} and \eqref{euler34} allow us to deduce the following identity:
$$\Lu_{1}^{(\alpha)}(x)=\Lu^{-1}\Lu_{2}^{(\alpha)}(x)\,\Lu.$$

\begin{coro}
\label{teogeneuler9}
The  inverses of the polynomial matrices $\Lu_{1}\left(x+\frac{1}{2}\right)$ and $\Lu_{2}\left(x+\frac{1}{2}\right)$  can be factorized as follows.
\begin{equation}
\label{euler35}
\left[\Lu_{1}\left(x+\frac{1}{2}\right)\right]^{-1}=\D P[-x]\Lu,
\end{equation}
\begin{equation}
\label{euler35a}
\left[\Lu_{2}\left(x+\frac{1}{2}\right)\right]^{-1}=\Lu \D P[-x].
\end{equation}
In particular,
\begin{equation}
\label{euler35b}
 \Lu_{1}^{-1}= \D P\left[\frac{1}{2}\right]\Lu, \quad \mbox{ and }  \quad  \Lu_{2}^{-1}= \Lu\D P\left[\frac{1}{2}\right],
\end{equation}
\begin{equation}
\label{euler35c}
\left[\Lu_{1}\left(\frac{1}{2}\right)\right]^{-1}= \D\Lu, \quad \mbox{ and }  \quad \left[ \Lu_{2}\left(\frac{1}{2}\right)\right]^{-1}= \Lu\D.
\end{equation}
\end{coro}

\begin{remark}
It is worthwhile to mention that if we consider $a\in \CC$, $b\in \CC\setminus\{0\}$ and $s=0,1$, then  Theorems \ref{teogeneuler5} and \ref{teogeneuler8}, as well as, their corollaries have corresponding analogous forms for generalized Fibonacci matrices of type $s$,  $\F^{(a,b,s)}$,  and for generalized Fibonacci matrices $\U^{(a,b,0)}$ with second order recurrent sequence $U_{n}^{(a,b)}$ subordinated to certain constraints. The reader may consult  \cite{SNS2008} in order to complete the details of this assertion.
\end{remark}

We finish this section with some new identities involving the Fibonacci numbers, the Lucas numbers and the generalized Euler polynomials and numbers.

\begin{teo}
\label{teogeneuler7}
For $0\leq r\leq n$ and $\alpha$ any real or complex number, we have
\begin{eqnarray}
\nonumber
\label{euler27}
\binom{n}{r}E^{(\alpha)}_{n-r}(x)&=& F_{n-r+1}+\left[(r+1)x -\frac{(r+1)\alpha +2}{2}\right]F_{n-r}\\
\nonumber
& & + \sum_{k=r+2}^{n}\binom{k}{r}\left\{E^{(\alpha)}_{k-r}(x)- \frac{k-r}{k}\left[E^{(\alpha)}_{k-r-1}(x)+ \frac{k-r-1}{k-1}E^{(\alpha)}_{k-r-2}(x)\right]\right\}F_{n-k+1}\\
\nonumber
\label{euler31}
& & \\
\nonumber
& & \\
\nonumber
&=& F_{n-r+1}+\left[n\left(x-\frac{\alpha}{2}\right)-1\right]F_{n-r}\\
\nonumber
& & + \sum_{k=0}^{n-2}\binom{n}{k}\left\{E^{(\alpha)}_{n-k}(x)- \frac{n-k}{k+1}\left[E^{(\alpha)}_{n-k-1}(x)+ \frac{n-k-1}{k+2}E^{(\alpha)}_{n-k-2}(x)\right]\right\}F_{k-r+1}.
 \end{eqnarray}
  \end{teo}

\begin{proof}
We proceed as in the proof of  \cite[Theorem 4.2]{ZW2006}, making the corresponding modifications. From \eqref{euler28a}, it is clear that
$$\tilde{m}_{r,r}^{(\alpha)}(x)= 1,  \quad \tilde{m}_{r+1,r}^{(\alpha)}(x)= (r+1)x -\frac{(r+1)\alpha +2}{2},$$
and, for $k\geq 2$:
$$\tilde{m}_{k,r}^{(\alpha)}(x)= \binom{k}{r}\left\{E^{(\alpha)}_{k-r}(x)- \frac{k-r}{k}\left[E^{(\alpha)}_{k-r-1}(x)+ \frac{k-r-1}{k-1}E^{(\alpha)}_{k-r-2}(x)\right]\right\}.$$

Next, it follows from \eqref{euler22} that
\begin{eqnarray}
\nonumber
\binom{n}{r}E^{(\alpha)}_{n-r}(x)&=& E^{(\alpha)}_{n,r}(x)\\
\nonumber
&=&\sum_{k=r}^{n}F_{n-k+1}\tilde{m}_{k,r}^{(\alpha)}(x)\\
\nonumber
&=&F_{n-r+1}+ F_{n-r}\tilde{m}_{r+1,r}^{(\alpha)}(x) +\sum_{k=r+2}^{n}F_{n-k+1}\tilde{m}_{k,r}^{(\alpha)}(x)\\
\nonumber
&=&F_{n-r+1}+\left[(r+1)x -\frac{(r+1)\alpha +2}{2}\right]F_{n-r}\\
\nonumber
& & + \sum_{k=r+2}^{n}\binom{k}{r}\left\{E^{(\alpha)}_{k-r}(x)- \frac{k-r}{k}\left[E^{(\alpha)}_{k-r-1}(x)+ \frac{k-r-1}{k-1}E^{(\alpha)}_{k-r-2}(x)\right]\right\}F_{n-k+1}.
 \end{eqnarray}
This chain of equalities completes the first part of the proof. The second one is obtained in a similar way, taking into account the following identities:
$$\tilde{n}_{n,n}^{(\alpha)}(x)= 1,  \quad \tilde{n}_{n,n-1}^{(\alpha)}(x)= n\left(x-\frac{\alpha}{2}\right)-1,$$
and, for $0\leq k\leq n-2$:
$$\tilde{n}_{n,k}^{(\alpha)}(x)= \binom{n}{k}\left\{E^{(\alpha)}_{n-k}(x)- \frac{n-k}{k+1}\left[E^{(\alpha)}_{n-k-1}(x)+ \frac{n-k-1}{k+2}E^{(\alpha)}_{n-k-2}(x)\right]\right\}.$$
\end{proof}

\begin{coro}
\label{corgeneuler5}
For $0\leq r\leq n$ and $\alpha$ any real number, we have
\begin{multline*}
\nonumber
(-1)^{n}\binom{n}{r}E^{(\alpha)}_{n-r}(x)=(-1)^{r}F_{n-r+1}+(-1)^{r+1}\left[\frac{(r+1)(2x-\alpha)+2}{2}\right]F_{n-r}\\
\hspace{2.5cm}+ \sum_{k=r+2}^{n}(-1)^{k}\binom{k}{r}\left\{E^{(\alpha)}_{k-r}(x)+ \frac{k-r}{k}\left[E^{(\alpha)}_{k-r-1}(x)- \frac{k-r-1}{k-1}E^{(\alpha)}_{k-r-2}(x)\right]\right\}F_{n-k+1}\\
\\
\nonumber = (-1)^{r}F_{n-r+1}+(-1)^{r+1}\left[n\left(x-\frac{\alpha}{2}\right)-1\right]F_{n-r}\\
 \hspace{2cm} + \sum_{k=0}^{n-2}(-1)^{n-k+r}\binom{n}{k}\left\{E^{(\alpha)}_{n-k}(x)+ \frac{n-k}{k+1}\left[E^{(\alpha)}_{n-k-1}(x)+ \frac{n-k-1}{k+2}E^{(\alpha)}_{n-k-2}(x)\right]\right\}F_{k-r+1}.
 \end{multline*}
\end{coro}

\begin{proof}
Replacing $x$ by $\alpha-x$ in \eqref{euler27} and applying the formula
$$E^{(\alpha)}_{n}(x)= (-1)^{n}E^{(\alpha)}_{n}(\alpha-x)$$
to the resulting identity, we obtain the first identity of Corollary \ref{corgeneuler5}. An analogous reasoning yields the second identity.
\end{proof}

Analogous reasonings to those used in the proofs of Theorem \ref{teogeneuler7} and Corollary \ref{corgeneuler5}  allow us to prove the following results.

\begin{teo}
\label{teogeneuler10}
For any real or complex number $\alpha$, we have the following identities
\begin{eqnarray}
\nonumber
\label{lucas3}
E^{(\alpha)}_{n}(x)&=& L_{n+1}+\left(x-\frac{\alpha}{2}-3\right)L_{n}+ \sum_{k=2}^{n}\left(E^{(\alpha)}_{k}(x)-3E^{(\alpha)}_{k-1}(x) \right)L_{n-k+1}\\
& & + 5\sum_{k=2}^{n}\sum_{s=0}^{k-2}(-1)^{k-s}2^{k-s-2}L_{n-k+1}E^{(\alpha)}_{s}(x),
\end{eqnarray}
whenever $n\geq 2$.
\begin{eqnarray}
\label{lucas4}
\nonumber
nE^{(\alpha)}_{n-1}(x)&=& L_{n}+\left(2x-\alpha-3\right)L_{n-1}+ \sum_{k=3}^{n}\left(kE^{(\alpha)}_{k-1}(x)-3(k-1)E^{(\alpha)}_{k-2}(x) \right)L_{n-k+1}\\
& & + 5\sum_{k=3}^{n}\sum_{s=1}^{k-2}(-1)^{k-s}2^{k-s-2}sL_{n-k+1}E^{(\alpha)}_{s-1}(x),
\end{eqnarray}
whenever $n\geq3$.
\end{teo}

\begin{coro}
The following identities hold.
\begin{eqnarray}
\nonumber
\label{lucas5}
(-1)^{n}E^{(\alpha)}_{n}(x)&=& L_{n+1}-\left(x-\frac{\alpha}{2}+3\right)L_{n}+ \sum_{k=2}^{n}(-1)^{k}\left(E^{(\alpha)}_{k}(x)+3E^{(\alpha)}_{k-1}(x) \right)L_{n-k+1}\\
& & + 5\sum_{k=2}^{n}\sum_{s=0}^{k-2}(-1)^{k-s}2^{k-s-2}L_{n-k+1}E^{(\alpha)}_{s}(x),
\end{eqnarray}
whenever $n\geq 2$.
\begin{eqnarray}
\label{lucas6}
\nonumber
(-1)^{n-1}nE^{(\alpha)}_{n-1}(x)&=& L_{n}+\left(\alpha-2x -3\right)L_{n-1}\\
\nonumber
&& + \sum_{k=3}^{n}(-1)^{k-1}\left(kE^{(\alpha)}_{k-1}(x)+3(k-1)E^{(\alpha)}_{k-2}(x) \right)L_{n-k+1}\\
& & + 5\sum_{k=3}^{n}\sum_{s=1}^{k-2}(-1)^{k-1}2^{k-s-2}sL_{n-k+1}E^{(\alpha)}_{s-1}(x),
\end{eqnarray}
whenever $n\geq3$.
\end{coro}

By \eqref{lucas3}, \eqref{lucas4}, \eqref{lucas5}  and \eqref{lucas6} we obtain the following  interesting identities involving Lucas and Euler numbers.

\begin{itemize}
\item For $n\geq 2$:
$$
E_{n}-\left(L_{n+1}-\frac{7}{2}L_{n}\right)= \sum_{k=2}^{n}\left(E_{k}-3E_{k-1}+5\sum_{s=0}^{k-2}(-1)^{k-s}2^{k-s-2}E_{s}\right)L_{n-k+1},
$$
$$
(-1)^{n}E_{n}= L_{n+1}-\frac{5}{2}L_{n}+ \sum_{k=2}^{n}(-1)^{k}\left(E_{k}+3E_{k-1}\right)L_{n-k+1}+ 5\sum_{k=2}^{n}\sum_{s=0}^{k-2}(-1)^{k-s}2^{k-s-2}L_{n-k+1}E_{s}.
$$
\item For $n\geq 3$:
\begin{eqnarray*}
nE_{n-1}-(L_{n}-4L_{n-1})&=& \sum_{k=3}^{n}\left(kE_{k-1}-3(k-1)E_{k-2}\right)L_{n-k+1}\\
 & & + 5\sum_{k=3}^{n}\sum_{s=1}^{k-2}(-1)^{k-s}2^{k-s-2}sE_{s-1}L_{n-k+1},
\end{eqnarray*}
$$(-1)^{n-1}nE_{n-1}-L_{n}-\left(\alpha-2x -3\right)L_{n-1} = \sum_{k=3}^{n}(-1)^{k-1}\left(kE_{k-1}+3(k-1)E_{k-2} \right)L_{n-k+1}$$
$$\qquad\qquad\hspace{4cm}\qquad + 5\sum_{k=3}^{n}\sum_{s=1}^{k-2}(-1)^{k-1}2^{k-s-2}sL_{n-k+1}E_{s-1}.$$
\end{itemize}

Another similar combinatorial identities may be obtained using the results of  \cite{LKC2003}. We leave to the interested reader the formulation of them.

\section{Euler matrices and their relation with Stirling and  Vandermonde matrices}
\label{sec:5}

Let $s(n, k)$ and $S(n, k)$ be the Stirling numbers of the first and second kind, which are respectively defined by the generating functions \cite[Chapther 1, Section 1.6]{SCh2012}:

\begin{eqnarray}
\label{seuler1}
\sum_{k=0}^{n}s(n, k)z^{k}&=& z(z-1)\cdots(z-n+1), \\
\nonumber
(\log(1+z))^{k}&=&k!\sum_{n=k}^{\infty}s(n, k)\frac{z^{n}}{n!}, \quad |z|<1,\\
\nonumber
z^{n}&=&  \sum_{k=0}^{n}S(n, k)z(z-1)\cdots(z-k+1),\\
\nonumber
(e^{z}-1)^{k} &=& k!\sum_{n=k}^{\infty}S(n, k)\frac{z^{n}}{n!}.
\end{eqnarray}

In combinatorics, it is well-known that the value $|s(n, k)|$ represents the number of permutations of $n$  elements with $k$ disjoint cycles. While, the Stirling numbers of the second kind $S(n, k)$ give the number of partitions of $n$ objects into $k$ non-empty subsets. Another way to compute these numbers is by means of the formula (see \cite[Eq. (5.1)]{E2012} or \cite[p. 226]{Rio68}:
\begin{equation*}
S(n, k)= \frac{1}{k!}\sum_{l=0}^{k}(-1)^{k-l}\binom{k}{l}l^{n}, \quad 1\leq k\leq n.
\end{equation*}

A recent connection between the Stirling numbers of the second kind and the Euler polynomials is given by the formula (see \cite[Theorem 3.1, Eq. (3.3)]{GQ2014}):
\begin{equation}
\label{seuler11}
E_{n}(x)=\sum_{k=0}^{n}(-1)^{n-k}\binom{n}{k}\left[\sum_{l=1}^{n-k+1}\frac{(-1)^{l-1}(l-1)!}{2^{l-1}}S(n-k+1,l)\right]x^{k}.
\end{equation}

Proceeding as in the proof of  \cite[Theorem 3.1]{GQ2014}, one can find a similar relation to the previous one but  connecting Stirling numbers of the first kind and a particular class of generalized Euler polynomials.

\begin{teo}
\label{stinglemma}
Let us assume that $\alpha=m\in \NN$. Then, the connection between the Stirling numbers of the first kind  and the generalized Euler polynomial $E^{(m)}_{n}(x)$ is given by the formula:
\begin{equation}
\label{sneuler1}
E^{(m)}_{n}(x)=\frac{1}{2^{n}}\sum_{k=0}^{n}\binom{n}{k}\left[\sum_{j=0}^{n-k}s(n-k,j)(-m)^{j}\right](2x)^{k},
\end{equation}
\end{teo}

\begin{proof}
By Leibniz's theorem for differentiation we have
\begin{eqnarray*}
\frac{\partial^{r}}{\partial z^{r}}\left[\left(\frac{2}{e^{z}+1}\right)^{m}e^{xz}\right]&=& \sum_{k=0}^{r}\binom{r}{k}\left[\left(\frac{2}{e^{z}+1}\right)^{m}\right]^{(k)}\frac{\partial^{r-k}}{\partial z^{r-k}}(e^{xz})\\
&=&\sum_{k=0}^{r}\binom{r}{k}\left[\left(\frac{2}{e^{z}+1}\right)^{m}\right]^{(k)}x^{r-k}e^{xz}\\
&=&\left(\frac{2}{e^{z}+1}\right)^{m}e^{(x+1)z}\sum_{k=0}^{r}\binom{r}{k}\frac{(-m)_{k}\,x^{r-k}}{(e^{z}+1)^{k}},
\end{eqnarray*}
where in the last expression $(-m)_{k}$ denotes the falling factorial with opposite argument $-m$.

Combining this with the $r$-th differentiation on both sides of the generating function in \eqref{euler1} reveals that
$$\sum_{n=r}^{\infty}E_{n}^{(m)}(x)\frac{z^{n-r}}{(n-r)!}=
\left(\frac{2}{e^{z}+1}\right)^{m}e^{(x+1)z}\sum_{k=0}^{r}\binom{r}{k}\frac{(-m)_{k}\,x^{r-k}}{(e^{z}+1)^{k}}.$$

Further taking $z\rightarrow 0$ and employing \eqref{seuler1} give
\begin{eqnarray*}
E_{r}^{(m)}(x)&=&\sum_{k=0}^{r}\binom{r}{k}(-m)_{k}\,\frac{x^{r-k}}{2^{k}}=\frac{1}{2^{r}}\sum_{k=0}^{r}\binom{r}{k}
(-m)_{r-k}\,(2x)^{k}\\
&=&\frac{1}{2^{r}}\sum_{k=0}^{r}\binom{r}{k}\left[\sum_{j=0}^{r-k}s(r-k, j)(-m)^{j} \right](2x)^{k}.
\end{eqnarray*}

Finally, changing $r$ by $n$ the proof of the formula \eqref{seuler11} is complete.
\end{proof}

\begin{definition}
\label{def6}
For the Stirling numbers $s(i,j)$ and $S(i,j)$ of the first kind and of the second kind
respectively, define $\st$ and $\St$ to be the $(n+1)\times(n+1)$ matrices by
\begin{equation}
\label{seuler10}
\st_{i,j}=
\left\{\begin{array}{l}
s(i,j), \quad i\geq j,\\
0, \quad \mbox{otherwise},
\end{array}\right.   \quad \mbox{ and } \quad
\St_{i,j}=
\left\{\begin{array}{l}
S(i,j), \quad i\geq j,\\
0, \quad \mbox{otherwise}.
\end{array}\right.
\end{equation}

The matrices $\st$ and $\St$  are called Stirling matrix of the first kind and of the second
kind, respectively (see \cite{ChK2001}).
\end{definition}

In order to obtain factorizations for Euler matrices via Stirling matrices we will need the following matrices:

Let $\tilde{S}_{n}$ be the factorial Stirling matrix, i.e., the $n\times n$ matrix whose $(i,j)$-th entry is given by $\tilde{S}_{i,j,n}:=j!\,\St_{i,j}$, $i\geq j$ and otherwise $0$.

For $m\in\NN$, let  $\st^{(m)}$ be the $(n+1)\times(n+1)$ matrix whose $(i,j)$-entries are defined by
\begin{equation}
\label{kind1}
\st^{(m)}_{i,j}=\left\{\begin{array}{l}
\binom{i}{j}\sum_{k=0}^{i-j}s(i-j, k)(-m)^{k}, \quad i\geq j,\\
\\
0, \quad \mbox{otherwise}.
\end{array}\right.
\end{equation}

Let  $\tilde{C}$ and $\tilde{D}$  be the $(n+1)\times(n+1)$ matrices whose $(i,j)$-entries are defined by

\begin{equation}
\label{kind2}
\tilde{C}_{i,j}=\left\{\begin{array}{l}
\binom{i}{j} (-1)^{i-j}\sum_{k=0}^{i-j}\left(-\frac{1}{2}\right)^{k}\tilde{S}_{i-j-k,k,i-j}, \quad i\geq j,\\
\\
0, \quad \mbox{otherwise}.
\end{array}\right.
\end{equation}

\begin{equation}
\label{kind3}
\tilde{D}_{i,j}=\left\{\begin{array}{l}
\binom{i}{j} (-1)^{i-j}\sum_{k=0}^{i-j}\left(-\frac{1}{2}\right)^{k}\tilde{S}_{i-j-k,k+1,i-j}, \quad i\geq j,\\
\\
0, \quad \mbox{otherwise}.
\end{array}\right.
\end{equation}

The next theorem shows the corresponding factorizations of the generalized Euler matrix $\E^{(m)}$, $m\in\NN$, in terms of the Stirling matrices, when the expressions \eqref{seuler11} and \eqref{sneuler1} are incorporated.

\begin{teo}
\label{teogenseuler5}
For $m\in \NN$, the generalized Euler matrix $\E^{(m)}(x)$  can be factorized as follows.
\begin{equation}
\label{seuler22}
\E^{(m)}(x)= \st^{(m)}P[x].
\end{equation}
In the case of the Stirling matrix of the second kind, we have
\begin{equation}
\label{sneuler22b}
\E(x)=  (\tilde{C}+\tilde{D})P[x].
\end{equation}
Furthermore,
\begin{equation}
\label{sneuler22r}
\st^{(1)}= \tilde{C}+\tilde{D}.
\end{equation}
\end{teo}

\begin{proof}
For $m\in \NN$ and $i\geq j$, let $A_{i,j}^{(m)}(x)$ be the $(i,j)$-th entry of the matrix product $\st^{(m)}P[x]$, then
$$\begin{aligned}
A_{i,j}^{(m)}(x)=&\sum_{k=j}^{i}\st^{(m)}_{i,k}\,p_{k,j}(x)=\sum_{k=j}^{i}\binom{i}{k}\binom{k}{j}\left[\sum_{r=0}^{i-k}s(i-k, r)(-m)^{r}\right]x^{k-j}\\
=&\sum_{k=j}^{i}\binom{i}{j}\binom{i-j}{k-j}2^{j-i}\left[\sum_{r=0}^{i-k}s(i-k, r)(-m)^{r}\right](2x)^{k-j}\\
=&\frac{\binom{i}{j}}{2^{i-j}}\sum_{k=0}^{i-j}\binom{i-j}{k}\left[\sum_{r=0}^{i-j-k}s(i-j-k, r)(-m)^{r}\right](2x)^{k}=\binom{i}{j}E_{i-j}^{(m)}(x),
\end{aligned}$$
The last equality is an immediate consequence of  \eqref{sneuler1}, and \eqref{seuler22} follows from the previous chain of equalities.

In order to prove \eqref{sneuler22b}, we proceed in a similar way to the previous one,  tanking into account  \eqref{seuler11}, \eqref{kind2}, \eqref{kind3}, and making the corresponding modifications.  Finally, the substitution $m=1$ into \eqref{seuler22} yields \eqref{sneuler22r}.
\end{proof}

\begin{definition}
\label{shift1}
The  $(n+1)\times(n+1)$ shifted Euler polynomial matrix $\tilde{\E}(x)$  are given by
\begin{equation}
\label{shift2}
\tilde{\E}_{i,j}(x)= \E_{i}(j+x), \quad 0\leq i,j\leq n.
\end{equation}
\end{definition}

Let us consider  the Vandermonde matrix:
$$\V(x):= \begin{bmatrix}
            1 & 1 & 1 & \cdots & 1 \\
            x & 1+x & 2+x & \cdots & n+x \\
            x^{2} &(1+x)^{2} & (2+x)^{2} & \cdots & (n+x)^{2} \\
            \vdots & \vdots & \vdots& \ddots & \vdots \\
            x^{n}& (1+x)^{n} & (2+x)^{n} & \cdots & (n+x)^{n}
          \end{bmatrix}.$$

In \cite[Theorem 2.1]{ChK2002}, the following factorization for the Vandermonde matrix $\V(x)$ was stated.
\begin{equation}
\label{veuler3}
\V(x)=  ([1]\oplus \tilde{S}_{n}) \Delta_{n+1}(x) P^{T}:=([1]\oplus \tilde{S}_{n}) \Delta_{n+1}(x) (P[1])^{T},
\end{equation}
where $\Delta_{n+1}(x)(P[1])^{T}$ represents the LU-factorization of a lower triangular matrix whose $(i,j)$-th entry is $\binom{x}{i-j}$, if $i\geq j$ and otherwise $0$.

The relation between the shifted Euler polynomial matrix $\tilde{\E}(x)$ and the matrices $\V(x)$ and  $\tilde{S}_{n}$ is contained in the following result.

\begin{teo}
\label{teogenvaneuler5}
The  shifted Euler polynomial matrix $\tilde{\E}(x)$  can be factorized in terms of the Vandermonde matrix $\V(x)$ and consequently, in terms of the factorial Stirling matrix $\tilde{S}_{n}$   as follows:
\begin{equation}
\label{veuler22}
\tilde{\E}(x)= \E \V(x),
\end{equation}
\begin{equation}
\label{veuler22c}
\tilde{\E}(x)= \E  ([1]\oplus \tilde{S}_{n}) \Delta_{n+1}(x) P^{T}.
\end{equation}
\end{teo}

\begin{proof}
Let $\tilde{\E}_{i,j}(x)$ be the $(i,j)$-th entry of the shifted Euler polynomial matrix $\tilde{\E}(x)$. Then, using \eqref{euler9} we get
$$\tilde{\E}_{i,j}(x)= \E_{i}(j+x)=\sum_{k=0}^{i}\binom{i}{k}E_{i-k}(j+x)^{k}=\sum_{k=0}^{i}E_{i,k}\V_{k,j}(x).$$
Hence, \eqref{veuler22} follows from this chain of equalities. The relation \eqref{veuler22c} is a straightforward consequence of \eqref{veuler3}.
\end{proof}

\begin{remark}
Note that the relations \eqref{veuler22} and \eqref{veuler22c} are the analogous of \cite[Eqs. (37), (38)]{ZW2006}, respectively, in the context of Euler polynomial matrices.
\end{remark}

Finally, in the present paper, all matrix identities have been expressed using finite matrices. Since such matrix identities involve lower triangular matrices, they have a resemblance for infinite matrices. We state this property briefly as follows.

Let $\E^{(\alpha)}_{\infty}(x)$, $\E_{\infty}\left(x+\frac{1}{2}\right)$, $\E_{\infty}(x)$, $\E_{\infty}$, $\EE_{\infty}$, $\D_{\infty}$, $\tilde{\E}_{\infty}(x)$, $P_{\infty}[x]$, $\F_{\infty}$, $\F^{-1}_{\infty}$, $\Lu_{\infty}$, $\G_{\infty}[x]$, $\HH_{\infty}[x]$, $\M^{(\alpha)}_{\infty}(x)$,  $\M_{\infty}\left(x+\frac{1}{2}\right)$, $\N^{(\alpha)}_{\infty}(x)$, $\N_{\infty}\left(x+\frac{1}{2}\right)$, $\V_{\infty}$ and $\st^{(m)}_{\infty}$, be the infinite cases of the matrices $\E^{(\alpha)}(x)$, $\E\left(x+\frac{1}{2}\right)$, $\E(x)$, $\E$, $\EE$, $\D$, $\tilde{\E}(x)$, $P[x]$, $\F$, $\F^{-1}$, $\Lu$, $\G[x]$, $\HH[x]$, $\M^{(\alpha)}(x)$,  $\M\left(x+\frac{1}{2}\right)$, $\N^{(\alpha)}(x)$, $\N\left(x+\frac{1}{2}\right)$, $\V$ and $\st^{(m)}$ respectively. Then the following identities hold.
\begin{eqnarray*}
2\E^{(\alpha-1)}_{\infty}(x)&=& \E^{(\alpha)}_{\infty}(x+1)+\E^{(\alpha)}_{\infty}(x),\\
\E^{(\alpha)}_{\infty}(x+y) &=& \E^{(\alpha)}_{\infty}(x)P_{\infty}[y]=P_{\infty}[x]\E^{(\alpha)}_{\infty}(y)=\E^{(\alpha)}_{\infty}(y)P_{\infty}[x], \\
 \E_{\infty}(x+y) &=& P_{\infty}[x]  \E_{\infty}(y)=P_{\infty}[y]  \E_{\infty}(x), \\
\E_{\infty}(x) &=& P_{\infty}[x]\E_{\infty}, \\
\E_{\infty}\left(x+\frac{1}{2}\right) &=& P_{\infty}[x] \EE_{\infty}, \\
\left[\E_{\infty}\left(x+\frac{1}{2}\right)\right]^{-1} &=& \D_{\infty}P_{\infty}[-x], \\
\E_{\infty}\left(x+\frac{1}{2}\right) &=& \Lu_{\infty}\G_{\infty}[x]\EE_{\infty}=\HH_{\infty}[x]\Lu_{\infty}\EE_{\infty}, \\
 \E^{(\alpha)}_{\infty}(x)  &=& \F_{\infty}\M^{(\alpha)}_{\infty}(x)=\N^{(\alpha)}_{\infty}(x)\F_{\infty}, \\
 \M^{(\alpha)}_{\infty}(x) &=& \F^{-1}_{\infty}\N^{(\alpha)}_{\infty}(x)\F_{\infty},\\
  \left[\M_{\infty}\left(x+\frac{1}{2}\right)\right]^{-1} &=& \D_{\infty}P_{\infty}[-x] \F_{\infty},\\
  \left[\N_{\infty}\left(x+\frac{1}{2}\right)\right]^{-1} &=& \F_{\infty} \D_{\infty}P_{\infty}[-x],\\
\E^{(m)}_{\infty}(x)&=&\st_{\infty}P_{\infty}[x],\\
\tilde{\E}_{\infty}(x)&=& \E_{\infty} \V_{\infty}(x).
\end{eqnarray*}

\bigskip





\begin{thebibliography}{}



\bibitem{AIK2014} T. Arakawa, T. Ibukiyama, M. Kaneko, {\it  Bernoulli Numbers and Zeta Functions},  Springer, New York (2014).

\bibitem{CV1993} G.S. Call, D.J.  Velleman, Pascal's matrices, {\it Amer. Math. Monthly} {\bf 100}, 372--376 (1993).

\bibitem{ChK2001} G.-S. Cheon, J.-S. Kim, Stirling matrix via Pascal matrix, {\it Linear Algebra Appl.}  {\bf 329}, 49--59 (2001).

\bibitem{ChK2002} G.-S. Cheon, J.-S. Kim, Factorial Stirling matrix and related combinatorial sequences, {\it  Linear Algebra Appl.} {\bf 357}, 247--258 (2002).


\bibitem{C1974} L. Comtet, {\it Advanced Combinatorics. The Art of Finite and Infinite Expansions}, D. Reidel Publishing Co., Dordrecht (1974).


\bibitem{E2012} T. Ernst, {\it A comprehensive treatment of $q$-calculus}, Birkh\"auser, Boston (2012).

\bibitem{GQ2014} B.-N. Guoa, F. Qi,  Explicit formulae for computing Euler polynomials in terms of Stirling numbers of the second kind, {\it J. Comput. Appl. Math.} {\bf  272}, 251--257 (2014).

\bibitem{HAS2016} Y. He, S. Araci, H.M. Srivastava, Some new formulas for the products of the Apostol type polynomials, {\it Adv. Differ. Equ.} {\bf 2016(Article ID 287)}, 1--18 (2016).

\bibitem{HASA2015} Y. He, S. Araci, H.M. Srivastava, M. Acikg\"{o}z, Some new identities for the Apostol-Bernoulli polynomials and the Apostol-Genocchi polynomials, {\it Appl. Math. Comput.} {\bf 262}, 31--41 (2015).

\bibitem{HQU2015} P. Hern\'andez-Llanos, Y. Quintana, A.  Urieles, About extensions of generalized Apostol-type polynomials, {\it   Results Math.} {\bf  68},  203--225 (2015).

\bibitem{IRS} G.I. Infante, J.L. Ram\'{\i}rez, A. \c{S}ahin, Some results on $q$-analogue of the Bernoulli, Euler and Fibonacci matrices, {\it Math. Rep. (Bucur.)} {\bf 19(69)} No. 4, 399--417 (2017).


\bibitem{LKL} G.-Y. Lee, J.-S. Kim, S.-G. Lee, Factorizations and eigenvalues of Fibonacci and symmetric Fibonacci matrices, {\it Fibonacci Quart.} {\bf 40}(3), 203--211 (2002).

\bibitem{LKC2003} G.-Y. Lee, J.-S. Kim, S.-H. Cho,  Some combinatorial identities via Fibonacci numbers, {\it  Discrete Appl. Math.} {\bf 130}(3), 527--534 (2003).

\bibitem{LS2005} Q.-M. Luo,  H.M. Srivastava, Some generalizations of the Apostol-Bernoulli and Apostol-Euler polynomials, {\it  J. Math. Anal. Appl.} {\bf 308}(1), 290--302 (2005).

\bibitem{N} N.E. N{\o}rlund, {\it Vorlesungen \"{u}ber Differenzenrechnung}, Springer-Verlag, Berlin (1924) (reprinted 1954), (in German).

\bibitem{PS2013} \'{A}. Pint\'er, H.M. Srivastava, Addition theorems for the Appell polynomials and the associated classes of polynomial expansions, {\it Aequationes Math.} {\bf 85}, 483--495 (2013).

\bibitem{QRU} Y. Quintana, W. Ram\'{\i}rez, A. Urieles, On an operational matrix method based on generalized Bernoulli polynomials of level $m$, {\it Calcolo} {\bf 55}(3) 29 pages, (2018).

\bibitem{Rio68} J. Riordan, {\it Combinatorial Identities}, Wiley, New York, (1968).

\bibitem{SBR2018} H.M. Srivastava, M.A. Boutiche, M. Rahmani, A class of Frobenius-type Eulerian polynomials, {\it Rocky Mountain J. Math.} {\bf 48}(3), 1003--1013  (2018).

\bibitem{SCh2012} H.M. Srivastava, J. Choi, {\it Zeta and $q$-Zeta Functions and Associated Series and Integrals}, Elsevier, London (2012).

\bibitem{SKS2017} H.M. Srivastava, I. Kucuko\u{g}lu, Y. Simsek, Partial differential equations for a new family of numbers and
polynomials unifying the Apostol-type numbers and the Apostol-type polynomials, {\it J. Number Theory} {\bf 181}, 117--146 (2017).

\bibitem{SM} H.M. Srivastava, H.L. Manocha, {\it A Treatise on Generating Functions}, Ellis Horwood Ltd., West Sussex (1984).

\bibitem{SMR2018} H.M. Srivastava, M. Masjed-Jamei, M. Reza Beyki, A parametric type of the Apostol-Bernoulli, Apostol-Euler and Apostol-Genocchi polynomials, {\it Appl. Math. Inform. Sci.} {\bf 12}(5), 907--916, (2018).

\bibitem{SOK2013} H.M. Srivastava, M.A. \"{O}zarslan, C. Kaanuglu, Some generalized Lagrange-based Apostol-Bernoulli,
Apostol-Euler and Apostol-Genocchi polynomials, {\it Russian J. Math. Phys.} {\bf 20}, 110--120 (2013).

\bibitem{SOY2014} H.M. Srivastava, M.A. \"{O}zarslan, B. Yilmaz, Some families of differential equations associated with the Hermite-based Appell polynomials and other classes of Hermite-based polynomials, {\it Filomat}  {\bf 28}(4), 695--708 (2014).

\bibitem{SP2003} H.M. Srivastava, \'{A}. Pint\'er, Remarks on some relationships between the Bernoulli and Euler polynomials, {\it  Appl. Math. Lett.}  {\bf 17}, 375--380 (2004).

\bibitem{SNS2008} P. Stanimirovi\'{c}, J.  Nikolov, I.  Stanimirovi\'{c}, A generalization of Fibonacci and Lucas matrices, {\it Discrete Appl. Math.}  {\bf 156}(14), 2606--2619 (2008).

\bibitem{Z1997} Z.Z. Zhang, The linear algebra of generalized Pascal matrix, {\it Linear Algebra Appl.} {\bf 250}, 51--60 (1997).

\bibitem{Z1998} Z.Z. Zhang, M.X. Liu, An extension of generalized Pascal matrix and its algebraic properties, {\it Linear Algebra Appl.}  {\bf 271}, 169--177 (1998).

\bibitem{ZW2006} Z. Zhang, J. Wang, Bernoulli matrix and its algebraic properties, {\it Discrete Appl. Math.}  {\bf 154}, 1622--1632  (2006).

 \bibitem{ZZ2007} Z. Zhang, Y. Zhang, The Lucas matrix and some combinatorial identities, {\it Indian J. Pure Appl. Math.}  {\bf 38}(5), 457--465 (2007).


\end{thebibliography}
\end{document}